\documentclass[
 aip,
 jmp,
 amsmath,amssymb,
 reprint,
 onecolumn,
 article-title,
 longbibliography
]{revtex4-2}

\usepackage[T1]{fontenc}
\usepackage{mathptmx}
\usepackage{amsthm,mathtools}
\usepackage{enumitem}

\allowdisplaybreaks

\newtheorem{theorem}{Theorem}
\newtheorem{proposition}[theorem]{Proposition}
\newtheorem{corollary}[theorem]{Corollary}
\newtheorem{lemma}[theorem]{Lemma}
\theoremstyle{definition}
\newtheorem{definition}[theorem]{Definition}
\theoremstyle{remark}
\newtheorem{remark}[theorem]{Remark}

\newcommand{\R}{\mathbb{R}}
\newcommand{\Kb}{\widehat K_{\beta}}

\newcommand{\1}{\mathbf{1}}

\begin{document}

\title{Strict unimodality of the exact time-domain diffuse-light response under a Robin boundary condition}

\author{Junyong Eom}
\email{eom@gunma-u.ac.jp}
\affiliation{Division of Pure and Applied Science, Graduate School of Science and Technology, Gunma University, Gunma 376-8515, Japan}

\author{Jinwan Park}
\affiliation{Department of Applied Mathematics, Kongju National University, Gongju-si, Chungcheongnam-do 32588, Republic of Korea}

\date{27 August 2026}

\begin{abstract}
Time-domain boundary measurements of diffuse light are temporal point-spread functions whose peak position is used as a robust scalar datum in diffuse optical and fluorescence imaging.  We consider the exact response generated by a single fluorescent point target in a homogeneous half-space with a Robin boundary condition.  The response is a temporal convolution of two diffusion Green-function factors containing a complementary-error-function boundary correction.  For every positive diffusion, absorption, depth, and source--target and detector--target distance parameter, and for every nonnegative Robin parameter, we prove that the normalized exact response is a generalized gamma convolution density.  Its extension by zero to the negative half-line is bell-shaped.  A differentiable small-time expansion, established by sectorial analyticity and Cauchy estimates, together with the sign-change structure of the first two derivatives, excludes tangential critical points and yields a unique peak with strict increase before it and strict decrease after it.  The argument is exact and requires no large-distance, source--detector balance, symmetry, restricted-time, or small-parameter assumption.  The same conclusion holds after convolution with an exponential fluorescence-lifetime density for every positive lifetime.  The key ingredients are an explicit Stieltjes representation of the Robin factor, hyperbolic complete monotonicity, closure of the generalized gamma convolution class, and bell-shape theory.  A derivation of the exact half-space model and the differentiated endpoint asymptotics are included in two appendices.
\end{abstract}


\maketitle

\noindent\textit{Preprint submitted to the Journal of Mathematical Physics.}\par\smallskip

\section{Introduction}

Time-domain boundary measurements of diffuse light record the temporal point-spread function produced by a short optical pulse in a highly scattering medium.  In the diffusion regime, this function is generated by a parabolic boundary-value problem and carries information about optical coefficients, propagation distances, and, in fluorescence measurements, the location and lifetime of fluorescent agents.  Robin boundary conditions arise from refractive-index mismatch at the tissue boundary and are standard in diffusion models of photon transport \cite{HaskellSvaasandTsayFengMcAdamsTromberg1994,MachidaNakamura2020}.  The broader mathematical and computational foundations of diffuse optical tomography are reviewed in \cite{Arridge1999,ArridgeSchotland2009}, while time-resolved forward and inverse formulations include finite-element, transform, temporal-moment, and feature-data approaches \cite{GaoNiuZhaoZhang1998,ZhangGaoHeZhao2008,HervePuszkaPlanatChretienDinten2012,MozumderTarvainen2020}.

We focus on a single fluorescent point target in the half-space.  After all time-independent physical constants are collected into a positive factor $C_*$, the exact zero-fluorescence-lifetime response has the convolution form
\begin{equation}\label{eq:intro-convolution}
 u_m(t)=C_*\,(h_{A_d}*h_{A_s})(t),
 \qquad t>0,
\end{equation}
where
\begin{equation}\label{eq:intro-hA}
 h_A(t)=e^{-kt}t^{-3/2}e^{-A/t}\Kb(a;t),
 \qquad A>0,
\end{equation}
and
\begin{equation}\label{eq:intro-Robin}
 \Kb(a;t)
 =1-\beta\sqrt{\pi qt}\,
 \exp\!\left(\frac{(a+2\beta qt)^2}{4qt}\right)
 \operatorname{erfc}\!\left(
 \frac{a+2\beta qt}{\sqrt{4qt}}
 \right).
\end{equation}
Here $q>0$ and $k>0$ are effective diffusion and absorption parameters, $a>0$ is the target depth, and $A_s,A_d>0$ encode the source--target and detector--target distances.  The boundary parameter satisfies $\beta\geq0$.  The value $\beta=0$ gives the Neumann case, for which $\widehat K_0\equiv1$; the case $\beta>0$ retains the genuine Robin correction.  Formula \eqref{eq:intro-convolution} is the full exact temporal response considered in this paper.  No asymptotic replacement has been made in \eqref{eq:intro-convolution}--\eqref{eq:intro-Robin}.  Appendix~\ref{app:model} derives these formulas from the Robin half-space Green function and the coupled excitation--emission point-target model; see also Refs.~\onlinecite{HaskellSvaasandTsayFengMcAdamsTromberg1994,MachidaNakamura2020,EomMachidaNakamuraNishimuraSun2023}.

\subsection*{Peak time as an inverse datum}
Among the scalar quantities extracted from a time-resolved trace, the peak time is particularly simple:
\begin{equation}\label{eq:intro-peak-time}
 t_{\rm peak}:=\operatorname*{arg\,max}_{t>0}u_m(t).
\end{equation}
Experimentally, it is obtained by locating the largest sampled fluorescence intensity, possibly after standard temporal calibration or smoothing, and therefore uses data near the dominant part of the signal.  It does not require numerical integration of the entire trace, as temporal moments do, nor the evaluation of several transformed data values.  Temporal moments, Fourier data, and Mellin--Laplace data are all useful feature measurements in time-domain optical tomography \cite{HervePuszkaPlanatChretienDinten2012,MozumderTarvainen2020}; peak-based reconstruction provides a complementary reduction to a directly interpretable time coordinate.  Peak-based data have consequently been used in experimental and numerical diffuse-light reconstruction.  In particular, the temporal position of the maximum can be independent of multiplicative intensity factors and can be extracted from the high-signal portion of the trace, features that motivate its use as a robust datum \cite{HallMaLesageWang2004,NishimuraTamura2005,RileyHassanChernomordikGandjbakhche2007}.  Peak-time reconstruction and analysis in the point-target setting were subsequently developed in \cite{SunNakamuraNishimuraJiangLiuMachida2020,SunZhang2022,ChenEomNakamuraNishimura2023,ChenEomNakamuraNishimura2026}.

The use of \eqref{eq:intro-peak-time} in an inverse problem requires a mathematical point that is easy to overlook.  The symbol $\operatorname*{arg\,max}$ denotes a single scalar datum only when the exact model response has a unique maximizer.  If several competing local or global maxima were possible, the reported peak could depend on temporal resolution, noise, preprocessing, or a convention for selecting one maximum.  The resulting parameter-to-peak-time relation would then not be intrinsically single-valued.  Thus strict unimodality is not merely a qualitative statement that a computed curve ``looks peaked''; it gives a basic well-posedness property for the reduced forward map
\[
 (q,k,a,A_s,A_d,\beta)\longmapsto t_{\rm peak}.
\]
This question is different from uniqueness for an inverse problem based on the full boundary time trace.  Full-data uniqueness and numerical inversion results are available for time-domain FDOT \cite{SunZhang2022,LiSun2026}, but they do not by themselves imply that the scalar peak-time observable of the exact Robin response is uniquely defined.

\subsection*{The solvable Neumann benchmark and the Robin difficulty}
The contrast with the Neumann case is instructive.  When $\beta=0$, the boundary correction disappears and the convolution in \eqref{eq:intro-convolution} can be evaluated explicitly by the Laplace transform, as shown in Ref.~\onlinecite{EomMachidaNakamuraNishimuraSun2023}.  Up to a positive multiplicative constant,
\begin{equation}\label{eq:intro-Neumann}
 u_m^{(0)}(t)
 =e^{-kt}t^{-3/2}
 \exp\!\left[-\frac{(\sqrt{A_d}+\sqrt{A_s})^2}{t}\right].
\end{equation}
Indeed, the elementary identity
\begin{equation}\label{eq:intro-Laplace}
 \int_0^\infty t^{-3/2}e^{-A/t}e^{-st}\,dt
 =\sqrt{\frac{\pi}{A}}e^{-2\sqrt{As}},
 \qquad A,s>0,
\end{equation}
turns the temporal convolution into a product of exponentials in $\sqrt{s}$.  If
\[
 B:=(\sqrt{A_d}+\sqrt{A_s})^2,
\]
then
\[
 \frac{d}{dt}\log u_m^{(0)}(t)
 =-k-\frac{3}{2t}+\frac{B}{t^2},
\]
and the critical-point equation is
\begin{equation}\label{eq:intro-Neumann-critical}
 kt^2+\frac32t-B=0.
\end{equation}
It has exactly one positive solution,
\begin{equation}\label{eq:intro-Neumann-peak}
 t_{\rm peak}^{(0)}
 =\frac{-3+\sqrt{9+16kB}}{4k}.
\end{equation}
Thus the exact uniqueness of the Neumann peak follows directly from a closed formula.

For a genuine Robin boundary condition, this argument is no longer available.  The complementary error-function correction \eqref{eq:intro-Robin} occurs in both propagation factors inside the convolution.  Although the response remains explicit as a one-dimensional integral, its Laplace transform does not lead to a comparably simple time-domain formula, and direct differentiation does not reduce the critical-point condition to a tractable algebraic equation.  The derivative contains several integral terms with no transparent common sign.  Positivity of the individual factors gives no immediate control over the number of zeros of the derivative of their convolution.  Consequently, the visual presence of a dominant numerical peak, or uniqueness of the maximizer of an approximate profile, does not by itself rule out additional critical points of the full response.

\subsection*{Earlier peak-time analysis}
A sequence of earlier works developed the analytical basis for using peak times in time-domain diffuse optical and fluorescence measurements.  For zero lifetime, the Neumann response was evaluated explicitly and its peak time was shown to be unique.  For Robin and Dirichlet-related boundary regimes, peak-time expressions were obtained mainly through asymptotic or approximate reductions.  In the Robin case, depending on the result under consideration, the analysis uses conditions such as a large transformed parameter, source--detector balance, or limiting boundary regimes \cite{EomMachidaNakamuraNishimuraSun2023}.  These results identify important exactly solvable and asymptotically tractable cases, but they should not be read as an unconditional strict-unimodality theorem for the full Robin convolution.

Subsequent work retained the Robin correction in a large-depth asymptotic temporal profile.  Analysis of the logarithmic derivative of that profile, equivalently of an associated cubic stationary equation, proved uniqueness of the approximate peak time.  The corresponding derivative asymptotics for the exact response were obtained in a time window scaled with the target depth and under a near-balance condition on the source--target and detector--target distances.  Numerical results support agreement between exact and approximate peak times in the corresponding regimes \cite{ChenEomNakamuraNishimura2023}.  The uniqueness result is therefore first a theorem for the asymptotic profile, with differentiated asymptotics connecting that profile to the exact response under the stated geometric and scaling hypotheses.

The same research program also produced inverse-problem results that are central to the interpretation of peak data but logically distinct from global unimodality.  Short-time asymptotic expansions were used to extract source--target--detector distance information and to study the sensitivity matrix, local identifiability, local Lipschitz stability, and iterative reconstruction for a point target \cite{EomNakamuraNishimuraSun2024}.  More recently, large-distance asymptotics yielded explicit approximate peak-time relations and a direct inversion scheme, with analyses covering zero and nonzero fluorescence lifetime \cite{ChenEomNakamuraNishimura2026}.  These works show that peak times encode useful geometric information and provide the mathematical foundation for peak-time inversion under assumptions tailored to the corresponding approximation.

The distinction relevant here is the following.  Previous works established exact peak-time uniqueness in special solvable cases, most notably the Neumann case, or uniqueness at the level of reduced asymptotic profiles.  For the genuine Robin response, the available analyses generally use large depth or distance, near balance of the source and detector geometry, restricted time scaling, limiting boundary parameters, or lifetime regimes appropriate to the asymptotic formula.  Such results do not directly provide a theorem that the full exact convolution \eqref{eq:intro-convolution} is strictly unimodal for arbitrary admissible geometry and arbitrary Robin parameter.

\subsection*{Main results}
The purpose of this paper is to establish that exact statement.  We do not infer the critical points of $u_m$ from a relative approximation.  Instead, we identify a probability-theoretic structure intrinsic to the Robin factor and the full temporal convolution.

\begin{definition}[Bell-shaped function]\label{def:bell-shaped-intro}
A nonnegative function $f\in C^\infty(\R)$ with $\lim_{|t|\to\infty}f(t)=0$ is called \emph{bell-shaped} if, for every integer $n\geq0$, the derivative $f^{(n)}$ changes sign exactly $n$ times on $\R$.
\end{definition}

\begin{theorem}[Exact generalized gamma convolution structure, bell-shape, and strict unimodality]\label{thm:intro-main}
Assume
\begin{equation}\label{eq:intro-parameter-range}
 q,k,a,A_d,A_s>0,
 \qquad \beta\geq0.
\end{equation}
Let $u_m$ be given by \eqref{eq:intro-convolution}--\eqref{eq:intro-Robin}, and set
\[
 M:=\int_0^\infty u_m(t)\,dt.
\]
Then $0<M<\infty$, and $M^{-1}u_m$ is the density of a generalized gamma convolution distribution on $[0,\infty)$.  Define the zero extension
\begin{equation}\label{eq:intro-zero-extension}
 \widetilde u_m(t):=
 \begin{cases}
  0,&t\leq0,\\
  u_m(t),&t>0.
 \end{cases}
\end{equation}
Then $\widetilde u_m\in C^\infty(\R)$ and is bell-shaped.  In particular, there is a unique $t_{\rm peak}>0$ such that
\[
 u_m'(t)>0\quad(0<t<t_{\rm peak}),
 \qquad
 u_m'(t_{\rm peak})=0,
 \qquad
 u_m'(t)<0\quad(t>t_{\rm peak}).
\]
\end{theorem}

Theorem~\ref{thm:intro-main} is exact rather than asymptotic.  It imposes no large-distance or large-depth condition, no source--detector symmetry or balance condition, no restricted time window, and no smallness or largeness assumption on $\beta$.  It also gives more than uniqueness of the maximum.  Bell-shape means that, for every $n\geq0$, the derivative $\widetilde u_m^{(n)}$ changes sign exactly $n$ times.  A sign-change statement alone does not automatically exclude additional tangential zeros of $u_m'$.  We therefore combine bell-shape with a differentiable small-time expansion, which gives $u_m'>0$ and $u_m''>0$ near the origin, and with analyticity.  These ingredients rule out tangential critical points and yield the strict derivative inequalities in the theorem.

A positive fluorescence lifetime is represented by convolution with an exponential waiting-time density.  Since this density belongs to the same convolution-stable probability class, the exact shape result persists without a lifetime restriction.

\begin{corollary}[Arbitrary positive fluorescence lifetime]\label{cor:intro-lifetime}
For every $\ell>0$, define
\begin{equation}\label{eq:intro-lifetime-response}
 U_{m,\ell}(t)
 :=\int_0^t \ell^{-1}e^{-(t-s)/\ell}u_m(s)\,ds,
 \qquad t>0.
\end{equation}
After normalization, $U_{m,\ell}$ is a generalized gamma convolution density.  Its zero extension, defined as in \eqref{eq:intro-zero-extension}, is bell-shaped, and $U_{m,\ell}$ is strictly unimodal for every $\ell>0$.
\end{corollary}

The proof is structural rather than perturbative.  We derive an exact Stieltjes representation of $\Kb(a;t)$ and prove that it is hyperbolically completely monotone.  This places each normalized factor $h_A$ in the class of generalized gamma convolution densities.  Closure of that class under convolution applies directly to \eqref{eq:intro-convolution}, and bell-shape theory controls the sign changes of all derivatives.  The differentiated endpoint expansion proved in Appendix~\ref{app:asymptotics} and a short analytic lemma then upgrade this control to a unique critical point.  Corollary~\ref{cor:intro-lifetime} follows from one further convolution with a gamma density and the corresponding endpoint signs.  Background on hyperbolically completely monotone (HCM) functions (Definition~\ref{def:HCM}) and generalized gamma convolution (GGC) distributions (Definition~\ref{def:GGC}) is given in \cite{Bondesson1992,BehmeBondesson2017}; the bell-shape step uses \cite{Kwasnicki2020}, with broader characterizations of bell-shaped functions developed in \cite{KwasnickiSimon2022}.

The remainder of the paper is organized as follows.  Section~II introduces the geometry, the exact response, its elementary analytic properties, and the explicit Neumann benchmark.  Section~III develops the Stieltjes representation of the Robin factor and records the HCM--GGC facts needed in the proof.  Section~IV proves Theorem~\ref{thm:intro-main} and Corollary~\ref{cor:intro-lifetime}, and then compares the exact conclusion with the uniqueness mechanism for a standard large-depth profile.  Appendix~\ref{app:model} derives the exact response from the Robin Green function, and Appendix~\ref{app:asymptotics} proves the differentiated small-time asymptotics.

\section{The exact temporal response}

This section isolates the forward model from the later probability argument.  We first specify the half-space geometry and the reduced parameters entering the temporal trace.  We then record positivity, integrability, smoothness, and endpoint behavior of the exact response.  The section closes with the Neumann calculation, which provides both a consistency check and a useful comparison with the non-elementary Robin case.

\subsection{Geometry and convolution factorization}
The half-space Green function used below is associated with the effective parameters $q=vD$ and $k=v\mu_a$ through the diffusion operator
\begin{equation}\label{eq:diffusion-operator}
 \partial_t-q\Delta+k
\end{equation}
and the homogeneous Robin condition $\partial_\nu G+\beta G=0$ on the planar boundary.  In the optical model, the Robin parameter represents the effective boundary interaction caused by refractive-index mismatch.  The derivation and interpretation of this boundary condition and related half-space Green functions may be found in \cite{HaskellSvaasandTsayFengMcAdamsTromberg1994,MachidaNakamura2020,EomMachidaNakamuraNishimuraSun2023}.  Appendix~\ref{app:model} records the reduction of the coupled excitation--emission point-target system to the exact convolution below.

Let
\[
 \Omega=\R^3_+,
 \qquad x_s,x_d\in\partial\Omega,
 \qquad x_c\in\Omega,
\]
where $x_s$ is the source, $x_d$ is the detector, and $x_c$ is the location of the fluorescent point target.  Write
\[
 a=(x_c)_3>0
\]
for the target depth and define the diffusive travel parameters
\begin{equation}\label{eq:parameters}
 q:=vD>0,
 \qquad k:=v\mu_a>0,
 \qquad A_s:=\frac{|x_s-x_c|^2}{4q},
 \qquad A_d:=\frac{|x_d-x_c|^2}{4q}.
\end{equation}
Because $x_s,x_d$ lie on the boundary and $x_c$ lies in the interior, $A_s,A_d>0$.  The quantities $A_s$ and $A_d$ are the diffusive travel-time parameters for the excitation and emission legs.  Their dependence on $x_c$ is the reason that a temporal feature such as $t_{\rm peak}$ can carry location information.

For $\beta\geq0$, define the Robin factor by
\begin{equation}\label{eq:Robin-factor}
 \Kb(a;t)
 :=1-\beta\sqrt{\pi qt}\,
 \exp\!\left(\frac{(a+2\beta qt)^2}{4qt}\right)
 \operatorname{erfc}\!\left(\frac{a+2\beta qt}{\sqrt{4qt}}\right),
 \qquad t>0.
\end{equation}
In the half-space Green function, the Gaussian factor $e^{-A/t}$ records diffusive travel from a boundary point to the target, while \eqref{eq:Robin-factor} accounts for the boundary interaction.  For $A>0$, introduce
\begin{equation}\label{eq:hA}
 h_A(t):=e^{-kt}t^{-3/2}e^{-A/t}\Kb(a;t),
 \qquad t>0.
\end{equation}
The zero-lifetime point-target response is the convolution already displayed in \eqref{eq:intro-convolution}.  In integral form,
\begin{align}
 u_m(t)
 ={}&C_*e^{-kt}\int_0^t ((t-s)s)^{-3/2}
 \exp\!\left(-\frac{A_d}{t-s}-\frac{A_s}{s}\right)\notag\\
 &\hspace{27mm}\times
 \Kb(a;t-s)\Kb(a;s)\,ds.
 \label{eq:exact-response-integral}
\end{align}
The variable $s$ is the excitation travel time to the target, and $t-s$ is the emission travel time from the target to the detector.  This decomposition is also the source of the convolution structure used later.

\subsection{Elementary analytic properties}
The following lemma collects facts that are needed when the probability classification is converted into a statement about derivatives on the whole real line.

\begin{lemma}[Positivity, integrability, flatness, and analyticity]\label{lem:basic-properties}
Let $A>0$ and assume $q,k,a>0$ and $\beta\geq0$.  Then $h_A$ is positive and belongs to $C^\infty(0,\infty)\cap L^1(0,\infty)$.  The exact response $u_m$ is real analytic on $(0,\infty)$; in fact, it extends analytically to the half-plane $\{t\in\mathbb C:\operatorname{Re}t>0\}$.  Moreover,
\begin{equation}\label{eq:flatness-hA}
 \lim_{t\downarrow0}h_A^{(n)}(t)=0
 \qquad\text{for every }n\geq0,
\end{equation}
and $h_A^{(n)}(t)$ decays exponentially as $t\to\infty$.  Consequently, the zero extension of $h_A$ belongs to $C^\infty(\R)$, and the same is true of the zero extension of $u_m$.
\end{lemma}

\begin{proof}
Positivity of $\Kb$ follows from the positive integral representation proved in Proposition~\ref{prop:stieltjes} below.  Smoothness on $(0,\infty)$ is immediate from \eqref{eq:Robin-factor}.  Moreover, the Stieltjes representation \eqref{eq:stieltjes-form} shows that every derivative of $\Kb(a;t)$ remains bounded as $t\downarrow0$: all positive moments of $\rho_\beta$ are finite because of the exponential decay in \eqref{eq:rho}.  Thus every derivative of $t^{-3/2}e^{-A/t}\Kb(a;t)$ is bounded by a negative power of $t$ multiplied by $e^{-A/t}$.  The exponential $e^{-A/t}$ dominates every power of $t^{-1}$, which gives \eqref{eq:flatness-hA}.  As $t\to\infty$, the factor $e^{-kt}$ dominates the algebraic factors and their derivatives.  Thus $h_A\in L^1(0,\infty)$ and its zero extension is smooth.  Analytic continuation of $u_m$ to $\operatorname{Re}t>0$, including an explicit endpoint-dominating function after the scaling $s=t\theta$, is proved in Lemma~\ref{lem:sectorial-analyticity} of Appendix~\ref{app:asymptotics}.  The convolution of the two smooth zero extensions is smooth on $\R$, proving the remaining assertion for $u_m$.
\end{proof}

\begin{corollary}[Existence of at least one peak]\label{cor:peak-existence}
Under the assumptions of Theorem~\ref{thm:intro-main}, the exact response satisfies
\[
 u_m(t)>0\quad(t>0),
 \qquad
 \lim_{t\downarrow0}u_m(t)=\lim_{t\to\infty}u_m(t)=0.
\]
Hence $u_m$ attains a global maximum at some point of $(0,\infty)$.
\end{corollary}

\begin{proof}
Positivity follows from \eqref{eq:intro-convolution}.  The endpoint limits follow
from Lemma~\ref{lem:basic-properties} and standard convolution estimates.
A positive continuous function vanishing at both endpoints attains its supremum
on a compact subinterval of $(0,\infty)$.
\end{proof}

Corollary~\ref{cor:peak-existence} proves only existence.  It gives no information about the number of local maxima.  The main theorem will supply exactly one critical point by using the global distributional structure of the normalized response.

\subsection{The Neumann case by the Laplace transform}
For $\beta=0$, the Robin factor is one and
\[
 h_A^{(0)}(t)=e^{-kt}t^{-3/2}e^{-A/t}.
\]
The explicit calculation can be recorded in the present notation as follows.

\begin{proposition}[Exact Neumann response]\label{prop:Neumann}
Let $\beta=0$ and set $B=(\sqrt{A_d}+\sqrt{A_s})^2$.  Then, for a constant $C_N>0$ independent of $t$,
\begin{equation}\label{eq:Neumann-response}
 u_m^{(0)}(t)=C_Ne^{-kt}t^{-3/2}e^{-B/t}.
\end{equation}
The response has exactly one critical point, given by
\begin{equation}\label{eq:Neumann-peak}
 t_{\rm peak}^{(0)}
 =\frac{-3+\sqrt{9+16kB}}{4k},
\end{equation}
and this point is the strict global maximum.
\end{proposition}

\begin{proof}
For $s>0$, identity \eqref{eq:intro-Laplace} gives
\[
 \mathcal L\{h_A^{(0)}\}(s)
 =\sqrt{\frac{\pi}{A}}\,
 e^{-2\sqrt{A(s+k)}}.
\]
The transform of the convolution is therefore a positive constant times
\[
 e^{-2(\sqrt{A_d}+\sqrt{A_s})\sqrt{s+k}}.
\]
Inverting the standard Laplace pair in \eqref{eq:intro-Laplace} gives \eqref{eq:Neumann-response}.  Differentiating its logarithm yields
\[
 \frac{(u_m^{(0)})'(t)}{u_m^{(0)}(t)}
 =-k-\frac{3}{2t}+\frac{B}{t^2}.
\]
The stationary equation is \eqref{eq:intro-Neumann-critical}, which has the unique positive root \eqref{eq:Neumann-peak}.  The logarithmic derivative is positive near zero and negative for large $t$, so the critical point is a strict global maximum.
\end{proof}

The calculation above is the exact mechanism available in the Neumann problem.  The remainder of the paper replaces explicit inversion by a structural argument that remains valid when $\beta>0$ and no elementary formula analogous to \eqref{eq:Neumann-response} is available.

\section{The Robin factor and the HCM--GGC framework}

The objective of this section is to convert the complementary-error-function expression \eqref{eq:Robin-factor} into a representation adapted to probability theory.  The key point is that the Robin factor is a Stieltjes transform of an explicit probability density.  This representation reveals hyperbolic complete monotonicity, which is precisely the property needed to place the factors $h_A$ in the generalized gamma convolution class.

\subsection{Complete monotonicity, HCM functions, and GGC laws}
We begin with the minimum background needed for the proof.  Detailed accounts can be found in \cite{Bondesson1992,SchillingSongVondracek2012,BehmeBondesson2017}.

\begin{definition}\label{def:CM-Stieltjes}
A function $f:(0,\infty)\to[0,\infty)$ is \emph{completely monotone} if $f\in C^\infty(0,\infty)$ and
\[
 (-1)^nf^{(n)}(t)\geq0
 \qquad(n=0,1,2,\ldots).
\]
A function $S$ is a \emph{Stieltjes function} if it admits a representation
\[
 S(t)=c_0+\frac{c_1}{t}
 +\int_{(0,\infty)}\frac{1}{t+\lambda}\,\mu(d\lambda)
\]
with $c_0,c_1\geq0$ and a positive measure $\mu$ satisfying
\[
 \int_{(0,\infty)}\frac{1}{1+\lambda}\,\mu(d\lambda)<\infty.
\]
\end{definition}

Every Stieltjes function is completely monotone.  In our application it is more convenient to use the equivalent normalized kernel $(1+\lambda t)^{-1}$, because it is the Laplace transform of an exponential random variable with random rate or scale.

\begin{definition}\label{def:HCM}
A nonnegative function $f$ on $(0,\infty)$ is \emph{hyperbolically completely monotone} (HCM) if, for every $u>0$, the function
\[
 v\longmapsto f(uv)f(u/v)
\]
is completely monotone as a function of $w=v+v^{-1}$ for $w\in(2,\infty)$, with the value at $w=2$ understood by continuity.
\end{definition}

The class of HCM functions is closed under multiplication and pointwise limits.  The basic factors
\begin{equation}\label{eq:basic-HCM}
 t^\gamma,
 \qquad e^{-ct},
 \qquad e^{-c/t}
 \qquad(\gamma\in\R,\ c\geq0)
\end{equation}
are HCM.  These examples already account for the non-Robin part of \eqref{eq:hA}.

\begin{definition}\label{def:GGC}
A probability distribution on $[0,\infty)$ is a \emph{generalized gamma convolution} (GGC) if it is a weak limit of finite convolutions of gamma distributions.  A probability distribution on $\R$ is an \emph{extended generalized gamma convolution} (EGGC) if it is a weak limit of finite convolutions of Gaussian distributions, gamma distributions, and reflected gamma distributions.  In particular, every GGC, viewed as a distribution on $\R$ supported in $[0,\infty)$, is an EGGC.
\end{definition}

The following standard implications are the bridge between analytic and probabilistic structure.
\begin{enumerate}[label=(P\arabic*),leftmargin=2.7em]
\item An HCM probability density defines a GGC distribution; see Proposition~3 of Ref.~\onlinecite{BehmeBondesson2017}, which cites Theorem~5.1.2 of Ref.~\onlinecite{Bondesson1992}.
\item A function $\phi$ with $\phi(0+)=1$ is the Laplace transform of a GGC distribution if and only if $\phi$ is HCM; see Proposition~4 of Ref.~\onlinecite{BehmeBondesson2017} and Theorem~6.1.1 of Ref.~\onlinecite{Bondesson1992}.
\item The GGC class is closed under convolution and weak limits.  The convolution closure follows directly from the defining approximation by finite convolutions of gamma laws.
\item If $X$ and $Y$ are independent positive random variables with HCM densities, then $XY$ and $X/Y$ have HCM densities; see Proposition~1 of Ref.~\onlinecite{BehmeBondesson2017}.
\end{enumerate}
We use (P4), followed by (P1) and (P2), in the proof of Proposition~\ref{prop:stieltjes}.

\subsection{Exact representations of the Robin correction}
The next proposition is the main analytic input.  Its first formula makes positivity immediate, the second expresses the Robin factor as an exponential-scale average, and the third is the Stieltjes representation used to prove the HCM property.

\begin{proposition}[Positive, exponential-scale, and Stieltjes representations]\label{prop:stieltjes}
For $a,q,t>0$ and $\beta\geq0$,
\begin{equation}\label{eq:positive-representation}
 e^{-a^2/(4qt)}\Kb(a;t)
 =\int_0^\infty e^{-\beta y}\frac{a+y}{2qt}
 \exp\!\left(-\frac{(a+y)^2}{4qt}\right)dy.
\end{equation}
For $\beta\geq0$, one also has
\begin{equation}\label{eq:inverse-laplace-form}
 \Kb(a;t)
 =\int_0^\infty \frac{1}{t}e^{-x/t}
 \exp\!\left[-\beta\bigl(\sqrt{a^2+4qx}-a\bigr)\right]dx
\end{equation}
or, equivalently, $\Kb(a;t)=\mathbb E[g(tE)]$, where $E$ is a standard exponential random variable and
\[
 g(x):=\exp\!\left[-\beta\bigl(\sqrt{a^2+4qx}-a\bigr)\right].
\]
If $\beta>0$, then the Stieltjes representation is
\begin{equation}\label{eq:stieltjes-form}
 \Kb(a;t)=\int_0^\infty\frac{1}{1+\lambda t}
 \rho_\beta(\lambda)\,d\lambda,
\end{equation}
where
\begin{equation}\label{eq:rho}
 \rho_\beta(\lambda)
 =\frac{\beta\sqrt q}{\sqrt\pi}e^{\beta a}\lambda^{-3/2}
 \exp\!\left(-\frac{\beta^2q}{\lambda}
 -\frac{a^2\lambda}{4q}\right),
 \qquad \lambda>0.
\end{equation}
Moreover, $\rho_\beta$ is an HCM probability density and $t\mapsto\Kb(a;t)$ is HCM.  For $\beta>0$, the Robin factor is strictly completely monotone, strictly decreasing, and log-convex.
\end{proposition}

\begin{proof}
Set
\[
 F_t(y):=\exp\!\left(-\frac{(a+y)^2}{4qt}\right).
\]
Since
\[
 -F_t'(y)=\frac{a+y}{2qt}F_t(y),
\]
integration by parts gives
\begin{align*}
 \int_0^\infty e^{-\beta y}\frac{a+y}{2qt}F_t(y)\,dy
 &=e^{-a^2/(4qt)}
 -\beta\int_0^\infty e^{-\beta y}F_t(y)\,dy.
\end{align*}
Completing the square in the final integral gives exactly the right-hand side of \eqref{eq:Robin-factor} multiplied by $e^{-a^2/(4qt)}$.  This proves \eqref{eq:positive-representation}, and in particular $\Kb(a;t)>0$.

In \eqref{eq:positive-representation}, use the change of variables
\[
 x=\frac{(a+y)^2-a^2}{4q},
 \qquad dx=\frac{a+y}{2q}\,dy.
\]
After multiplication by $e^{a^2/(4qt)}$, one obtains \eqref{eq:inverse-laplace-form}.

We next use the classical identity
\begin{equation}\label{eq:sqrt-laplace}
 e^{-c\sqrt s}
 =\frac{c}{2\sqrt\pi}\int_0^\infty
 \lambda^{-3/2}\exp\!\left(-\frac{c^2}{4\lambda}\right)
 e^{-s\lambda}\,d\lambda,
 \qquad c,s>0.
\end{equation}
Taking $c=2\beta\sqrt q$ and $s=x+a^2/(4q)$ yields
\begin{equation}\label{eq:rho-laplace}
 \exp\!\left[-\beta\bigl(\sqrt{a^2+4qx}-a\bigr)\right]
 =\int_0^\infty e^{-\lambda x}\rho_\beta(\lambda)\,d\lambda.
\end{equation}
At $x=0$, the left-hand side equals one, so $\rho_\beta$ integrates to one.  Substitution of \eqref{eq:rho-laplace} into \eqref{eq:inverse-laplace-form} and Tonelli's theorem give
\begin{align*}
 \Kb(a;t)
 &=\int_0^\infty\int_0^\infty
 \frac1t e^{-(1/t+\lambda)x}\rho_\beta(\lambda)
 \,d\lambda\,dx\\
 &=\int_0^\infty\frac{1}{1+\lambda t}
 \rho_\beta(\lambda)\,d\lambda,
\end{align*}
which is \eqref{eq:stieltjes-form}.

The density \eqref{eq:rho} is, up to a positive constant, of the form
\[
 \lambda^{-3/2}e^{-c/\lambda-d\lambda},
 \qquad c,d>0,
\]
and is HCM by the closure rules recalled above.  Let $\Lambda$ have density $\rho_\beta$, and let $E$ be an independent standard exponential random variable.  The exponential density $e^{-x}\1_{(0,\infty)}(x)$ is HCM.  By (P4), the product $E\Lambda$ has an HCM density; by (P1), its law is therefore GGC.  Conditioning on $\Lambda$ and using $\mathbb E[e^{-sE}]=(1+s)^{-1}$ for $s\geq0$, its Laplace transform is
\[
 \mathbb E[e^{-tE\Lambda}]
 =\mathbb E\left[\frac{1}{1+t\Lambda}\right]
 =\Kb(a;t).
\]
The Laplace transform of a GGC distribution is HCM, proving the HCM assertion for $\Kb$.

Finally, \eqref{eq:stieltjes-form} is a nontrivial positive mixture of the functions $(1+\lambda t)^{-1}$.  Differentiation under the integral gives
\[
 (-1)^n\frac{d^n}{dt^n}\Kb(a;t)
 =n!\int_0^\infty
 \frac{\lambda^n}{(1+\lambda t)^{n+1}}
 \rho_\beta(\lambda)\,d\lambda>0,
\]
which proves strict complete monotonicity and strict decrease.  Complete monotonicity implies log-convexity.
\end{proof}

\begin{remark}[What the representation accomplishes]\label{rem:representation-role}
The formula \eqref{eq:stieltjes-form} does not attempt to locate the mode directly.  Its role is more global: it identifies $\Kb$ as an HCM function and thereby places the entire factor \eqref{eq:hA} in a convolution-stable class.  This avoids the need to estimate the sign of the derivative of the integral in \eqref{eq:exact-response-integral} point by point.
\end{remark}

For $\beta=0$, $\widehat K_0\equiv1$ is HCM, so the structural conclusions below include the Neumann case without a separate limiting argument.

\subsection{The propagation factors are HCM densities}
We now combine the Robin representation with the elementary factors in \eqref{eq:basic-HCM}.

\begin{proposition}\label{prop:hA-HCM}
For every $A>0$, the function $h_A$ in \eqref{eq:hA} is HCM and integrable.  If
\begin{equation}\label{eq:normalization-hA}
 M_A:=\int_0^\infty h_A(t)\,dt,
 \qquad
 p_A(t):=M_A^{-1}h_A(t),
\end{equation}
then $p_A$ is an HCM probability density and hence the density of a GGC distribution.
\end{proposition}

\begin{proof}
The functions $e^{-kt}$, $t^{-3/2}$, and $e^{-A/t}$ are HCM.  Proposition~\ref{prop:stieltjes} shows that $\Kb(a;\cdot)$ is HCM for $\beta>0$, while it is identically one for $\beta=0$.  Closure of the HCM class under multiplication proves that $h_A$ is HCM.  Integrability and positivity were established in Lemma~\ref{lem:basic-properties}.  Normalization preserves HCM, and an HCM probability density defines a GGC law.
\end{proof}

Proposition~\ref{prop:hA-HCM} is the point at which the analytic half-space kernel enters the GGC class.  The remaining proof uses only closure under convolution and the smooth bell-shape theorem.

\section{Exact bell-shape, unimodality, and fluorescence lifetime}

This section assembles the preceding ingredients.  First, the exact zero-lifetime response is normalized and identified as a GGC density.  Second, smoothness of its zero extension allows the bell-shape theorem to be applied.  Third, the same argument is extended to every positive fluorescence lifetime.  We finish by comparing the exact conclusion with the stationary equation of a commonly used asymptotic profile.

\subsection{GGC classification and bell-shape}
Let $M_A$ and $p_A$ be given by \eqref{eq:normalization-hA}.  From \eqref{eq:intro-convolution},
\begin{equation}\label{eq:normalized-convolution}
 \frac{u_m(t)}{C_*M_{A_d}M_{A_s}}
 =(p_{A_d}*p_{A_s})(t).
\end{equation}
Each factor on the right is a GGC density by Proposition~\ref{prop:hA-HCM}.  Since the GGC class is closed under convolution, the right-hand side is a GGC density.  This proves the distributional assertion of Theorem~\ref{thm:intro-main} and identifies
\[
 M=C_*M_{A_d}M_{A_s}.
\]

The bell-shape step uses the EGGC consequence of Kwa\'snicki's Theorem~1.1 and the discussion following it, together with the smoothness implication in Sec.~2 of Ref.~\onlinecite{Kwasnicki2020}: every EGGC is weakly bell-shaped, and a nonzero weakly bell-shaped density that has a $C^\infty(\R)$ version vanishing at $\pm\infty$ is bell-shaped in the sense of Definition~\ref{def:bell-shaped-intro}.  Smoothness is not automatic for an arbitrary one-sided GGC density, so this hypothesis must be checked for the response itself.

By \eqref{eq:normalized-convolution}, the normalized law of $u_m$, viewed as a distribution on the real line supported in $[0,\infty)$, is a GGC and hence an EGGC by Definition~\ref{def:GGC}.  Lemma~\ref{lem:basic-properties} verifies directly that its density has the $C^\infty(\R)$ version $M^{-1}\widetilde u_m$, flat at the origin and vanishing at both ends.  Kwa\'snicki's theorem therefore applies, and $\widetilde u_m$ is bell-shaped.  In particular, $u_m'$ changes sign exactly once and $u_m''$ changes sign exactly twice.

The last statement gives ordinary unimodality, but a further argument is needed to exclude zeros of $u_m'$ at which no sign change occurs.  We now supply the endpoint information required for that step.

\subsection{Small-time derivative signs}
Set
\begin{equation}\label{eq:L}
 L:=\sqrt{A_d}+\sqrt{A_s}>0.
\end{equation}
The basic convolution identity
\begin{align}
 &\int_0^t (t-s)^{-3/2}s^{-3/2}
 \exp\!\left(-\frac{A_d}{t-s}-\frac{A_s}{s}\right)ds \notag\\
 &\qquad =\frac{\sqrt\pi\,L}{\sqrt{A_dA_s}}\,
 t^{-3/2}\exp\!\left(-\frac{L^2}{t}\right)
 \label{eq:stable-convolution}
\end{align}
follows by multiplying the two Laplace transforms in \eqref{eq:intro-Laplace}.  With $g$ as in Proposition~\ref{prop:stieltjes}, Lemma~\ref{lem:Robin-sector} proves the uniform estimate
\[
 g(x)=1-\frac{2\beta q}{a}x+O(x^2).
\]
The exponential-scale representation \eqref{eq:inverse-laplace-form}, together with
\[
 \int_0^\infty\frac1t e^{-x/t}x\,dx=t,
 \qquad
 \int_0^\infty\frac1t e^{-x/t}x^2\,dx=2t^2,
\]
therefore gives
\begin{equation}\label{eq:Robin-small-time}
 \Kb(a;t)=1-\frac{2\beta q}{a}t+O(t^2),
 \qquad t\downarrow0,
\end{equation}
and, uniformly for $0\leq s\leq t$,
\begin{equation}\label{eq:Robin-product-small-time}
 \Kb(a;s)\Kb(a;t-s)
 =1-\frac{2\beta q}{a}t+O(t^2).
\end{equation}
Combining \eqref{eq:exact-response-integral}, \eqref{eq:stable-convolution}, and \eqref{eq:Robin-product-small-time} yields
\begin{equation}\label{eq:um-small-time}
 u_m(t)=C_0t^{-3/2}e^{-L^2/t}\bigl(1+O(t)\bigr),
 \qquad t\downarrow0,
\end{equation}
for a constant $C_0>0$.  The differentiated form of this expansion is not inferred by formally differentiating an $O(t)$ remainder.  Lemmas~\ref{lem:Robin-sector}, \ref{lem:sectorial-analyticity}, and \ref{lem:Cauchy-derivatives} in Appendix~\ref{app:asymptotics} establish a uniform sectorial expansion and use Cauchy estimates to prove
\begin{equation}\label{eq:um-logder-small}
 \frac{u_m'(t)}{u_m(t)}
 =\frac{L^2}{t^2}-\frac{3}{2t}+O(1),
 \qquad t\downarrow0,
\end{equation}
and
\begin{equation}\label{eq:um-second-small}
 \frac{u_m''(t)}{u_m(t)}
 =\frac{L^4}{t^4}+O(t^{-3}),
 \qquad t\downarrow0.
\end{equation}
Consequently, there exists $\varepsilon>0$ such that
\begin{equation}\label{eq:initial-signs}
 u_m'(t)>0,
 \qquad
 u_m''(t)>0,
 \qquad 0<t<\varepsilon.
\end{equation}

\subsection{Exclusion of tangential critical points}
The following elementary lemma isolates the final step from bell-shape to strict unimodality.

\begin{lemma}[Bell-shape and strict unimodality]\label{lem:tangential-zero}
Let $f$ be positive and real analytic on $(0,\infty)$.  Suppose that its extension by zero is bell-shaped and that
\[
 f'(t)>0,
 \qquad f''(t)>0
\]
for all sufficiently small $t>0$.  Then $f'$ has exactly one zero.  It is positive before that zero and negative after it.
\end{lemma}

\begin{proof}
Bell-shape implies that $f'$ changes sign exactly once and $f''$ changes sign exactly twice.  Since $f'>0$ near zero, the analytic function $f'$ is not identically zero, and hence all of its zeros are isolated and have finite order.  Moreover, $f(t)\to0$ as $t\to\infty$, so the mean value theorem shows that $f'$ is negative somewhere.  The unique sign change of $f'$ is therefore from positive to negative; let $m$ denote its sign-changing zero.  Thus $f'\geq0$ on $(0,m)$ and $f'\leq0$ on $(m,\infty)$, with strict inequalities away from the isolated zeros.

At $m$, the sign change forces an odd-order zero.  Thus, for some integer $r\geq0$ and some $c_m>0$,
\[
 f'(m+x)=-c_m x^{2r+1}+O(x^{2r+2})
 \qquad (x\to0).
\]
Consequently, $f''(t)<0$ on a punctured neighborhood of $m$.

Suppose that $f'$ has a further zero $z<m$.  Then $f'$ is positive on both punctured sides of $z$, so the zero has even order with positive leading coefficient: for some $s\geq1$ and $c_z>0$,
\[
 f'(z+x)=c_zx^{2s}+O(x^{2s+1}).
\]
It follows that $f''<0$ immediately to the left of $z$ and $f''>0$ immediately to its right.  Together with $f''>0$ near zero and $f''<0$ near $m$, one can therefore choose four increasing points in $(0,m)$ at which the signs of $f''$ are, successively, $+,-,+,-$.  Thus $f''$ changes sign at least three times, contrary to bell-shape.

If a further zero $z>m$ exists, then $f'$ is negative on both punctured sides of $z$, and the even-order tangency has negative leading coefficient:
\[
 f'(z+x)=-c_zx^{2s}+O(x^{2s+1})
\]
for some $s\geq1$ and $c_z>0$.  Hence $f''>0$ immediately to the left of $z$ and $f''<0$ immediately to its right.  Using $f''>0$ near zero and $f''<0$ near $m$, we again obtain four increasing points at which the signs of $f''$ are $+,-,+,-$, so $f''$ has at least three sign changes.  This is impossible.  Therefore $m$ is the only zero of $f'$, and the asserted strict signs follow from the positive-to-negative sign change there.
\end{proof}

Apply Lemma~\ref{lem:tangential-zero} to $f=u_m$.  Real analyticity follows from Lemma~\ref{lem:sectorial-analyticity}, bell-shape was established above, and the initial signs are given by \eqref{eq:initial-signs}.  Thus $u_m'$ has exactly one zero and satisfies the strict sign inequalities in Theorem~\ref{thm:intro-main}.  This completes the proof of the zero-lifetime result.

\begin{remark}[Levels of shape information]\label{rem:three-levels}
It is useful to keep four statements separate.
\begin{enumerate}[label=(\roman*),leftmargin=2.2em]
\item A reduced or asymptotic profile has a stationary equation with one positive root.
\item The exact response has a unique global maximizer.
\item The exact response is strictly unimodal in the derivative sense: it increases strictly before its unique critical point and decreases strictly afterward.
\item The zero extension of the exact response is bell-shaped, so every derivative has the prescribed number of sign changes.
\end{enumerate}
Bell-shape gives the global sign-change count, while the endpoint signs and analyticity exclude tangential critical points and upgrade (iv) to (iii); statement (iii) then implies (ii).  Earlier asymptotic peak-time formulas often begin with (i), and differentiated estimates are needed to transfer information to the exact response in the relevant regime.
\end{remark}

\begin{remark}[Why a relative approximation is not enough]\label{rem:derivative-control}
A relation of the form
\[
 u_m(t)=u_m^{\mathrm a}(t)(1+o(1))
\]
controls function values but does not, by itself, preserve the number of critical points.  Small oscillatory perturbations can be negligible in relative amplitude while creating extra zeros of the derivative.  For example, on any compact interval where $f_0>0$, the functions $f_N(t)=f_0(t)(1+\varepsilon\sin Nt)$ satisfy $|f_N/f_0-1|\leq\varepsilon$, while the term $\varepsilon Nf_0(t)\cos Nt$ in $f_N'$ can generate many additional critical points when $N$ is large.  A perturbative proof of exact unimodality therefore requires derivative control away from the approximate mode and a nondegeneracy estimate near it.  The GGC argument bypasses this issue by classifying the exact convolution directly.
\end{remark}

\subsection{Positive fluorescence lifetime}
The fluorescence-lifetime response in \eqref{eq:intro-lifetime-response} is the convolution
\begin{equation}\label{eq:lifetime-convolution}
 U_{m,\ell}=f_\ell*u_m,
 \qquad
 f_\ell(t):=\ell^{-1}e^{-t/\ell}\1_{(0,\infty)}(t).
\end{equation}
The density $f_\ell$ is gamma with shape one and is therefore GGC.  The normalized zero-lifetime response is GGC by Theorem~\ref{thm:intro-main}; closure under convolution (P3) implies that the normalized $U_{m,\ell}$ is GGC for every $\ell>0$.

By Lemma~\ref{lem:basic-properties}, the zero extension of $u_m$ is flat at the origin.  The identity
\begin{equation}\label{eq:lifetime-ODE}
 \ell U_{m,\ell}'(t)+U_{m,\ell}(t)=u_m(t),
 \qquad U_{m,\ell}(0)=0,
\end{equation}
then shows inductively that the zero extension of $U_{m,\ell}$ is $C^\infty(\R)$.  It also vanishes at $\pm\infty$, so the same EGGC bell-shape theorem used above applies.

To obtain strict unimodality, we again verify the initial derivative signs.  The endpoint Laplace estimate obtained from \eqref{eq:um-small-time} is
\begin{equation}\label{eq:lifetime-small-time}
 U_{m,\ell}(t)
 =\frac{t^2}{\ell L^2}u_m(t)\bigl(1+O(t)\bigr),
 \qquad t\downarrow0.
\end{equation}
In particular, $U_{m,\ell}(t)/u_m(t)\to0$.  Equation \eqref{eq:lifetime-ODE} therefore gives $U_{m,\ell}'>0$ near zero.  Differentiating \eqref{eq:lifetime-ODE} and using \eqref{eq:um-logder-small} shows that $u_m'$ dominates $U_{m,\ell}'$ as $t\downarrow0$, and hence $U_{m,\ell}''>0$ near zero.  Since $u_m$ is real analytic on $(0,\infty)$, the linear ODE \eqref{eq:lifetime-ODE} implies that $U_{m,\ell}$ is real analytic there.  Lemma~\ref{lem:tangential-zero} therefore applies and proves that $U_{m,\ell}'$ has exactly one zero.  This proves Corollary~\ref{cor:intro-lifetime} for every $\ell>0$.

The lifetime result is therefore an exact closure statement, not a small- or large-lifetime approximation: adding an independent exponential waiting time preserves the GGC class and, together with the endpoint signs, preserves strict unimodality.

\subsection{Comparison with a reduced peak-time profile}
To clarify the difference between the exact theorem and asymptotic peak-time uniqueness, consider the representative large-depth profile
\begin{equation}\label{eq:asymptotic-profile}
 u_m^{\mathrm a}(t)
 =C_{\mathrm a}e^{-kt}t^{-3/2}e^{-\lambda^2/t}
 \left(\frac{a}{a+\beta qt}\right)^2,
 \qquad t>0,
\end{equation}
where $C_{\mathrm a},k,a,\lambda,q>0$ and $\beta\geq0$.  This profile retains a rational approximation of the Robin correction while replacing the full convolution by a single explicit expression.  Its logarithmic derivative is
\begin{equation}\label{eq:reduced-logderivative}
 \frac{d}{dt}\log u_m^{\mathrm a}(t)
 =-k-\frac{3}{2t}+\frac{\lambda^2}{t^2}
 -\frac{2\beta q}{a+\beta qt}.
\end{equation}
The stationary equation is equivalent to
\begin{equation}\label{eq:stationary-reduced}
 k t^2+\frac32t
 +\frac{2\beta q t^2}{a+\beta qt}
 -\lambda^2=0.
\end{equation}
The left-hand side is strictly increasing on $(0,\infty)$, is negative at $0+$, and tends to $+\infty$.  Thus \eqref{eq:asymptotic-profile} has a unique critical point, which is its strict global maximum.  This elementary monotonicity is one reason that the reduced profile is effective for deriving explicit or polynomial peak-time relations.

Theorem~\ref{thm:intro-main} concerns a different object: the exact convolution \eqref{eq:exact-response-integral}.  Its conclusion is valid before passing to \eqref{eq:asymptotic-profile} and without the geometric or scaling assumptions used to justify that approximation.  The two analyses are complementary.  The asymptotic profile supplies explicit relations useful for inversion, while the exact GGC theorem guarantees that the underlying Robin response possesses one and only one peak throughout the full admissible parameter range.

The proof may be summarized by the structural chain
\begin{align*}
 \text{Stieltjes representation of }\Kb
 &\Longrightarrow \Kb\text{ is HCM},\\
 &\Longrightarrow M_{A_d}^{-1}h_{A_d},M_{A_s}^{-1}h_{A_s}\text{ are HCM densities},\\
 &\Longrightarrow u_m\text{ is GGC up to normalization},\\
 &\Longrightarrow \widetilde u_m\text{ is bell-shaped},\\
 &\Longrightarrow u_m\text{ is strictly unimodal}.
\end{align*}
An additional convolution with the exponential density $f_\ell$ gives the same chain for every positive fluorescence lifetime.  This identifies a probability class behind peak-time uniqueness and separates the exact shape theorem from the asymptotic calculations used to obtain explicit inversion formulas.

\section*{CONCLUSION}
The exact time-domain boundary response generated by a single fluorescent point target in the half-space has one and only one peak for every admissible Robin parameter and every admissible source--target--detector geometry.  The result is not obtained by differentiating an approximate peak-time formula.  Instead, the Robin correction admits an explicit Stieltjes representation with an HCM mixing density, the two propagation factors become GGC densities after normalization, and convolution closure places the full response in the same class.  Bell-shape then controls all derivative sign changes, while the small-time derivative signs---proved by sectorial analyticity and Cauchy estimates---and real analyticity exclude tangential critical points.  The same mechanism remains valid for every positive fluorescence lifetime.  Thus the exact theorem supplies a model-level justification for treating the peak time as a single-valued boundary datum beyond the asymptotic regimes used to derive explicit inversion formulas.  Natural next questions are whether comparable HCM--GGC structures persist for more general geometries or optical models, and how the resulting peak-time map can be incorporated into quantitative stability and reconstruction analyses.  These extensions require separate arguments and are not consequences of the present single-point-target theorem.

\section*{ACKNOWLEDGMENTS}
Junyong Eom was supported by the JSPS KAKENHI (Grant Number 25K17291). Jinwan Park was supported by Basic Science Research Program through the National Research Foundation of Korea (NRF) funded by the Ministry of Education (RS-2025-25415411).

\section*{AUTHOR DECLARATIONS}
\subsection*{Conflict of Interest}
The authors have no conflicts to disclose.
\subsection*{Author Contributions}
Junyong Eom: Conceptualization; Investigation; Writing -- original draft; Writing -- review \& editing. Jinwan Park: Formal analysis; Validation; Methodology; Writing -- review \& editing.

\section*{DATA AVAILABILITY}
Data sharing is not applicable to this article as no new data were created or analyzed in this study.

\appendix

\section{Derivation of the exact model}\label{app:model}
We record the derivation of \eqref{eq:intro-convolution}--\eqref{eq:intro-Robin} from the Robin half-space Green function, so that the exactness of the forward model is documented within the paper.  The underlying optical modeling is discussed in Refs.~\onlinecite{HaskellSvaasandTsayFengMcAdamsTromberg1994,MachidaNakamura2020,EomMachidaNakamuraNishimuraSun2023}.

Let
\[
 \Gamma_k(x,t):=e^{-kt}(4\pi qt)^{-3/2}
 \exp\!\left(-\frac{|x|^2}{4qt}\right)
\]
be the whole-space fundamental solution of \eqref{eq:diffusion-operator}.  For the half-space $\Omega=\{x_3>0\}$ the outward normal is $\nu=-e_3$, so the Robin condition is $-\partial_{x_3}G+\beta G=0$ on $\{x_3=0\}$.  Writing $y^*:=(y_1,y_2,-y_3)$ for the reflected point, the Green function is
\begin{equation}\label{eq:Robin-Green}
 G(x,y,t)=\Gamma_k(x-y,t)+\Gamma_k(x-y^*,t)
 -2\beta\int_0^\infty e^{-\beta r}
 \Gamma_k(x-y^*+re_3,t)\,dr.
\end{equation}
Indeed, the right-hand side solves the diffusion equation and has the correct initial singularity.  At $x_3=0$, differentiation under the integral gives
\[
 \partial_{x_3}\Gamma_k(x-y^*+re_3,t)
 =\partial_r\Gamma_k(x-y^*+re_3,t),
\]
and integration by parts in $r$ verifies $\partial_{x_3}G=\beta G$, equivalently the stated outward-normal Robin condition.

Let $x\in\partial\Omega$ and $y=x_c$ with depth $a=(x_c)_3>0$.  If $\rho$ is the horizontal distance between $x$ and $x_c$, then
\[
 |x-y|^2=|x-y^*|^2=\rho^2+a^2,
 \qquad
 |x-y^*+re_3|^2=\rho^2+(a+r)^2.
\]
Completing the square,
\[
 \beta r+\frac{(a+r)^2}{4qt}
 =\frac{(r+a+2\beta qt)^2}{4qt}
 -\frac{(a+2\beta qt)^2-a^2}{4qt},
\]
so
\begin{align}
 \int_0^\infty e^{-\beta r}e^{-(a+r)^2/(4qt)}\,dr
 & =e^{((a+2\beta qt)^2-a^2)/(4qt)}
 \int_{a+2\beta qt}^\infty e^{-z^2/(4qt)}\,dz \notag\\
 & =\sqrt{\pi qt}\,e^{-a^2/(4qt)}
 e^{(a+2\beta qt)^2/(4qt)}
 \operatorname{erfc}\!\left(\frac{a+2\beta qt}{\sqrt{4qt}}\right).
 \label{eq:Robin-integral-eval}
\end{align}
Substitution into \eqref{eq:Robin-Green} gives, for $x\in\partial\Omega$,
\begin{equation}\label{eq:boundary-Green-hA}
 G(x,x_c,t)
 =\frac{2}{(4\pi q)^{3/2}}
 e^{-kt}t^{-3/2}
 e^{-|x-x_c|^2/(4qt)}\Kb(a;t)
 =\frac{2}{(4\pi q)^{3/2}}h_A(t),
\end{equation}
where $A=|x-x_c|^2/(4q)$.  Reciprocity gives the excitation field at the target due to the boundary source $x_s$ as a positive time-independent multiple of $h_{A_s}$, and the emission field at the boundary detector $x_d$ due to the target as a positive multiple of $h_{A_d}$.  In the zero-lifetime point-target model, the measured response is the temporal convolution of the two propagation legs,
\[
 u_m(t)=C_*\int_0^t h_{A_d}(t-s)h_{A_s}(s)\,ds,
\]
which is \eqref{eq:intro-convolution}.  All source strength, fluorescence yield, and time-independent Green-function prefactors are absorbed into $C_*>0$.

\section{Differentiated small-time asymptotics}\label{app:asymptotics}
This appendix proves the analyticity claim in Lemma~\ref{lem:basic-properties} and the differentiated expansions \eqref{eq:um-logder-small}--\eqref{eq:um-second-small}.  The argument extends \eqref{eq:um-small-time} uniformly to a complex sector and then applies Cauchy estimates.  Throughout the appendix, powers of $t$ on the right half-plane are taken with the principal branch.

\begin{lemma}[The Robin factor in the right half-plane]\label{lem:Robin-sector}
Let
\[
 \mathcal D_g:=\mathbb C\setminus\left(-\infty,-\frac{a^2}{4q}\right]
\]
and use the principal square root on $\mathbb C\setminus(-\infty,0]$.  Define
\[
 g(z):=\exp\!\left[-\beta\bigl(\sqrt{a^2+4qz}-a\bigr)\right],
 \qquad z\in\mathcal D_g.
\]
Then $g$ is holomorphic on $\mathcal D_g$, an open set containing the closed right half-plane, and $|g(z)|\leq e^{\beta a}$ for $\operatorname{Re}z\geq0$.  Moreover, there is $C_g>0$ such that
\begin{equation}\label{eq:g-quadratic-bound}
 \left|g(z)-1+\frac{2\beta q}{a}z\right|
 \leq C_g|z|^2,
 \qquad \operatorname{Re}z\geq0.
\end{equation}
If the continuous boundary value $\Kb(a;0):=1$ is used, then
\begin{equation}\label{eq:K-u-representation}
 \Kb(a;t)=\int_0^\infty e^{-u}g(tu)\,du,
 \qquad \operatorname{Re}t\geq0,
\end{equation}
where the right-hand side is holomorphic for $\operatorname{Re}t>0$, agrees with the original Robin factor for real $t>0$, and satisfies
\begin{equation}\label{eq:K-sector-expansion}
 \left|\Kb(a;t)-1+\frac{2\beta q}{a}t\right|
 \leq2C_g|t|^2,
 \qquad \operatorname{Re}t\geq0.
\end{equation}
For $\beta=0$ these statements hold with $g\equiv1$.
\end{lemma}

\begin{proof}
The affine map $z\mapsto a^2+4qz$ sends $\mathcal D_g$ into $\mathbb C\setminus(-\infty,0]$, so the stated branch makes $g$ holomorphic on $\mathcal D_g$.  If $\operatorname{Re}z\geq0$, then $\operatorname{Re}(a^2+4qz)\geq a^2>0$ and the principal square root has positive real part.  Hence $|g(z)|\leq e^{\beta a}$.

Set $c:=2\beta q/a$.  Since $g(0)=1$ and $g'(0)=-c$, the quotient
\[
 H(z):=\frac{g(z)-1+cz}{z^2}
\]
has a removable singularity at $z=0$.  It is therefore bounded on every compact subset of $\mathcal D_g$, in particular on $\{\operatorname{Re}z\geq0,\ |z|\leq1\}$.  For $\operatorname{Re}z\geq0$ and $|z|\geq1$,
\[
 \frac{|g(z)-1+cz|}{|z|^2}
 \leq \frac{e^{\beta a}+1}{|z|^2}+\frac{c}{|z|}
 \leq e^{\beta a}+1+c.
\]
Increasing the compact-set bound if necessary proves \eqref{eq:g-quadratic-bound} on the entire closed right half-plane.

For real $t>0$, \eqref{eq:K-u-representation} follows from \eqref{eq:inverse-laplace-form} by the substitution $x=tu$.  For $\operatorname{Re}t\geq0$, the integrand is bounded by $e^{-u}e^{\beta a}$.  On each compact subset of $\{\operatorname{Re}t>0\}$ it is holomorphic in $t$ and has the same integrable majorant, so dominated holomorphic parameter integration proves analyticity there.  Finally, integrating \eqref{eq:g-quadratic-bound} against $e^{-u}du$ and using $\int_0^\infty u^2e^{-u}du=2$ gives \eqref{eq:K-sector-expansion}, including the boundary value at $t=0$.
\end{proof}

\begin{lemma}[Sectorial analyticity and uniform expansion]\label{lem:sectorial-analyticity}
The response extends holomorphically to $\operatorname{Re}t>0$ through
\begin{align}
 u_m(t)={}&C_*e^{-kt}t^{-2}\int_0^1[\theta(1-\theta)]^{-3/2}
 \exp\!\left(-\frac{\phi(\theta)}{t}\right) \notag\\
 &\hspace{23mm}\times
 \Kb(a;t\theta)\Kb(a;t(1-\theta))\,d\theta,
 \label{eq:scaled-complex-response}
\end{align}
where
\begin{equation}\label{eq:phi-L-definition}
 \phi(\theta):=\frac{A_d}{1-\theta}+\frac{A_s}{\theta},
 \qquad
 L:=\sqrt{A_d}+\sqrt{A_s}.
\end{equation}
Thus $L^2=\min_{0<\theta<1}\phi(\theta)$.  For $\operatorname{Re}t>0$, define
\begin{align}
 J(t)&:=\int_0^1[\theta(1-\theta)]^{-3/2}
 e^{-\phi(\theta)/t}\,d\theta,
 \label{eq:J-definition}\\
 \mathcal R(t)&:=\int_0^1[\theta(1-\theta)]^{-3/2}
 e^{-\phi(\theta)/t}
 \bigl\{\Kb(a;t\theta)\Kb(a;t(1-\theta))-1\bigr\}\,d\theta,
 \label{eq:R-definition}\\
 E(t)&:=\frac{\mathcal R(t)}{J(t)}.
 \label{eq:E-definition}
\end{align}
For every $\delta\in(0,\pi/2)$ there are $t_0,C_\delta>0$ such that, on
\[
 S_\delta:=\{t\in\mathbb C:0<|t|\leq t_0,\ |\arg t|\leq\delta\},
\]
\begin{equation}\label{eq:sectorial-um-expansion}
 u_m(t)=C_0e^{-kt}t^{-3/2}e^{-L^2/t}\bigl(1+E(t)\bigr),
 \qquad |E(t)|\leq C_\delta|t|,
\end{equation}
where $C_0=C_*\sqrt\pi L/\sqrt{A_dA_s}$ and $E$ is holomorphic on $\operatorname{Re}t>0$.
\end{lemma}

\begin{proof}
The substitution $s=t\theta$ in \eqref{eq:exact-response-integral} gives \eqref{eq:scaled-complex-response} for real $t>0$.  Let $K\Subset\{\operatorname{Re}t>0\}$.  Since $\operatorname{Re}(1/t)$ is positive and continuous on $K$, there is $c_K>0$ such that $\operatorname{Re}(1/t)\geq c_K$ on $K$.  Lemma~\ref{lem:Robin-sector} gives a uniform bound for the Robin factors, while
\[
 \left|e^{-\phi(\theta)/t}\right|
 \leq e^{-c_K\phi(\theta)}.
\]
Because $\phi(\theta)\to\infty$ as $\theta\downarrow0$ or $\theta\uparrow1$, the function $[\theta(1-\theta)]^{-3/2}e^{-c_K\phi(\theta)}$ is integrable on $(0,1)$.  Dominated holomorphic parameter integration therefore proves that the right-hand side of \eqref{eq:scaled-complex-response}, as well as $J$ and $\mathcal R$, is holomorphic on the right half-plane.  Since it agrees with the original response on the positive real axis, it is the required analytic continuation.

Identity \eqref{eq:stable-convolution}, after $s=t\theta$, yields for real $t>0$
\begin{equation}\label{eq:J-exact}
 J(t)=\frac{\sqrt\pi L}{\sqrt{A_dA_s}}\,t^{1/2}e^{-L^2/t}.
\end{equation}
Both sides are holomorphic on the right half-plane, with the principal branch of $t^{1/2}$, and hence \eqref{eq:J-exact} holds there by the identity theorem.  In particular, $J$ has no zeros.  Thus $E=\mathcal R/J$ is holomorphic, and \eqref{eq:scaled-complex-response} becomes exactly
\[
 u_m(t)=C_*e^{-kt}t^{-2}J(t)\bigl(1+E(t)\bigr).
\]

By \eqref{eq:K-sector-expansion} and the uniform boundedness of $\Kb(a;t\alpha)$ for $0\leq\alpha\leq1$, after decreasing $t_0$ if necessary there is $C_1>0$ such that
\begin{equation}\label{eq:K-product-sector}
 \left|\Kb(a;t\theta)\Kb(a;t(1-\theta))-1\right|
 \leq C_1|t|
\end{equation}
uniformly for $\theta\in[0,1]$ and $t\in S_\delta$.  Put
\[
 \sigma:=\operatorname{Re}(1/t)=\frac{\cos(\arg t)}{|t|}.
\]
Using \eqref{eq:J-exact} at the positive real number $1/\sigma$ gives
\[
 \int_0^1[\theta(1-\theta)]^{-3/2}e^{-\sigma\phi(\theta)}\,d\theta
 =\frac{\sqrt\pi L}{\sqrt{A_dA_s}}\,\sigma^{-1/2}e^{-L^2\sigma}.
\]
Consequently,
\[
 |\mathcal R(t)|
 \leq C_1|t|\frac{\sqrt\pi L}{\sqrt{A_dA_s}}\,
 \sigma^{-1/2}e^{-L^2\sigma},
\]
whereas \eqref{eq:J-exact} gives
\[
 |J(t)|=\frac{\sqrt\pi L}{\sqrt{A_dA_s}}\,
 |t|^{1/2}e^{-L^2\sigma}.
\]
It follows that
\[
 |E(t)|\leq C_1|t|^{1/2}\sigma^{-1/2}
 =\frac{C_1|t|}{\sqrt{\cos(\arg t)}}
 \leq\frac{C_1}{\sqrt{\cos\delta}}|t|,
\]
which proves \eqref{eq:sectorial-um-expansion}.
\end{proof}

\begin{lemma}[Cauchy estimates and derivative signs]\label{lem:Cauchy-derivatives}
There are constants $C',C''>0$ and $t_1>0$ such that, for real $t\in(0,t_1)$,
\begin{equation}\label{eq:E-Cauchy-bounds}
 |E'(t)|\leq C',
 \qquad
 |E''(t)|\leq C''t^{-1}.
\end{equation}
Consequently, \eqref{eq:um-logder-small}, \eqref{eq:um-second-small}, and \eqref{eq:initial-signs} hold.
\end{lemma}

\begin{proof}
Fix $\delta_0=\pi/4$ in Lemma~\ref{lem:sectorial-analyticity} and choose $r=(\sin\delta_0)/2$.  The distance from the positive point $t$ to either boundary ray of the sector $|\arg w|<\delta_0$ is $t\sin\delta_0$, so $r<\sin\delta_0$ places the disk $\overline D(t,rt)$ inside that angular sector.  If additionally $(1+r)t\leq t_0$, then $|w|\leq t_0$ throughout the disk.  Hence, for all sufficiently small $t>0$, $\overline D(t,rt)\subset S_{\delta_0}$ and
\[
 |E(w)|\leq C_{\delta_0}|w|\leq C_{\delta_0}(1+r)t
 \qquad (w\in\overline D(t,rt)).
\]
Cauchy's estimates therefore give
\[
 |E'(t)|\leq\frac{C_{\delta_0}(1+r)t}{rt},
 \qquad
 |E''(t)|\leq\frac{2C_{\delta_0}(1+r)t}{(rt)^2},
\]
which proves \eqref{eq:E-Cauchy-bounds}.

Shrink $t_1$ if necessary so that $|E(t)|\leq1/2$ on $(0,t_1)$.  For positive real $t$, define
\[
 Q(t):=\frac{E'(t)}{1+E(t)}.
\]
The Cauchy bounds imply $Q(t)=O(1)$ and, without differentiating an unspecified remainder,
\[
 Q'(t)=\frac{E''(t)}{1+E(t)}
 -\frac{(E'(t))^2}{(1+E(t))^2}
 =O(t^{-1}).
\]
Logarithmic differentiation of the exact factorization \eqref{eq:sectorial-um-expansion} now gives
\[
 \frac{u_m'(t)}{u_m(t)}
 =-k-\frac{3}{2t}+\frac{L^2}{t^2}+Q(t)
 =\frac{L^2}{t^2}-\frac{3}{2t}+O(1),
\]
which is \eqref{eq:um-logder-small}.  Differentiating this exact logarithmic identity through $Q$ yields
\[
 \left(\frac{u_m'}{u_m}\right)'(t)
 =-\frac{2L^2}{t^3}+\frac{3}{2t^2}+Q'(t)
 =-\frac{2L^2}{t^3}+O(t^{-2}).
\]
Therefore
\begin{align*}
 \frac{u_m''(t)}{u_m(t)}
 &=\left(\frac{u_m'(t)}{u_m(t)}\right)^2
 +\left(\frac{u_m'}{u_m}\right)'(t)\\
 &=\left(\frac{L^2}{t^2}+O(t^{-1})\right)^2
 -\frac{2L^2}{t^3}+O(t^{-2})
 =\frac{L^4}{t^4}+O(t^{-3}).
\end{align*}
This proves \eqref{eq:um-second-small}; positivity of the leading terms in the two differentiated expansions gives \eqref{eq:initial-signs}.
\end{proof}

%

\end{document}